\documentclass[12pt]{article}

\usepackage{amsfonts,amsmath, amssymb,latexsym}
\def\Z{{\mathbb Z}}

\def\ct{{\rm ct}}

\def\Disc{{\rm Disc}}

\def\Aut{{\rm Aut}}

\def\R{{\mathbb R}}
\def\F{{\mathbb F}}

\def\Q{{\mathbb Q}}

\def\Z{{\mathbb Z}}

\def\F{{\mathbb F}}
\def\Q{{\mathbb Q}}
\def\C{{\mathbb C}}

\def\m1{{\textrm{\textbf{Id}}}}

\def\Aut{{\rm Aut}}

\def\Gal{{\mathrm{Gal}}}
\def\Aut{{\mathrm{Aut}}}

\def\Gal{{\mathrm{Gal}}}

\def\ind{{\mathrm{ind}}}

\def\ct{{\mathrm{ct}}}

\newtheorem{theorem}{Theorem}
\newtheorem{corollary}[theorem]{Corollary}

\newtheorem{lemma}[theorem]{Lemma}

\newtheorem{proposition}[theorem]{Proposition}
\newenvironment{proof}{\noindent {\bf Proof:}}{$\Box$ \vspace{2 ex}}

\title{A proof of van der Waerden's Conjecture on\\ random Galois groups of polynomials}  
\author{\,Manjul Bhargava\,\\ Princeton University }

\date{}

\begin{document}

\maketitle


\begin{abstract}
Of the $(2H+1)^n$ monic integer polynomials $f(x)=x^n+a_1 x^{n-1}+\cdots+a_n$ with $\max\{|a_1|,\ldots,|a_n|\}\leq H$, how many have associated Galois group that is not the full symmetric group $S_n$? There are clearly $\gg H^{n-1}$ such
polynomials, as may be obtained by setting $a_n=0$. In 1936, van der Waerden conjectured that $O(H^{n-1})$ should in fact also
be the correct upper bound for the count of such polynomials. The conjecture has been known previously for degrees $n\leq 4$,
due to work of van der Waerden and Chow and Dietmann.  

In this expository article, we outline a proof of van der Waerden's Conjecture for all degrees~$n$.\footnote{This article is the text of the announcement of  and talk on this result given on the occasion of Don~Zagier's 70th birthday celebration in the Number Theory Web Seminar on July 1, 2021.
Some details and arguments have been added for readability and completeness.  Happy birthday, Don!}
\end{abstract}

\section{Introduction}

Let $E_n(H)$ denote the number of monic integer polynomials $f(x)=x^n+a_1 x^{n-1} +\cdots+a_n$ of degree $n$ with $|a_i|\leq H$ for all $i$ such that the Galois group $\Gal(f)$ is not $S_n$. 
There are clearly $\gg H^{n-1}$ such polynomials,  as can be seen by setting $a_n=0$. 
In 1936, van der Waerden made the tantalizing conjecture that $O(H^{n-1})$ should in fact also be the correct upper bound for the count of such polynomials.
In other words, the probability that a monic polynomial with coefficients bounded by $H$ in absolute value has Galois group not isomorphic to $S_n$ \linebreak is~$\asymp 1/H$.  



Hilbert irreducibility implies that $E_n(H)=o(H^n)$, i.e., $100\%$ of monic polynomials of degree $n$ are irreducible and have Galois group $S_n$.  In 1936, van der Waerden~\cite{vdw} proved the first quantitative version of this statement by demonstrating that \vspace{-.05in}
$$E_n(H)= O(H^{n-\frac{1}{6(n-2)\log\log H}}).$$
The first power-saving bound was obtained by Knobloch~\cite{Knobloch} (1956)  who proved that
 $$E_n(H)= O(H^{n-\frac{1}{18n(n!)^3}}); $$ 
successive improvements to Knobloch's bound were then given by 
Gallagher~\cite{Gallagher} (1973) who proved using his large sieve that 
    $$E_n(H)= O(H^{n-1/2+\epsilon}),  $$ 
Zywina~\cite{Zywina} (2010) who using a ``larger sieve'' refined this to
    $$E_n(H)= O(H^{n-1/2}), $$  
Dietmann~\cite{Dietmann2} (2010) who proved using resolvent polynomials and the determinant method that 
 $$E_n(H)=O(H^{n-2+\sqrt2}), $$ and 
Anderson, Gafni, Lemke Oliver, Lowry-Duda, Shakan, and Zhang~\cite{aimgroup} (2021) who prove using a Selberg-style sieve that $$E_n(H)=O(H^{n-\frac23+\frac8{9n+21}+\epsilon}).$$
(For more on the uses of the large sieve in this and related problems, see the works of Cohen~\cite{Cohen} and Serre~\cite{Serre}.)



The purpose of this article is to prove that, indeed, 
$E_n(H)=O(H^{n-1})$, as was conjectured by van der Waerden: 
\begin{theorem}\label{main}
We have 
$E_n(H)=O(H^{n-1}).$
\end{theorem}

More generally, for any permutation group $G\subset S_n$ on $n$ letters, let $N_n(G,H)$ denote the number of monic integer polynomials $f(x)=x^n+a_1 x^{n-1} +\cdots+a_n$ with $|a_i|\leq H$ for all $i$ such that~$\Gal(f)\cong G$.  Then the above theorem amounts to proving that $N_n(G,H)=O(H^{n-1})$ for all permutation groups $G<S_n$. 

The methods we describe can in fact be used to give the best known bounds on $N_n(G,H)$ for various individual Galois groups $G$ (see~\cite{vdwproof} for details), and can also be used to prove a number of other variations of Theorem~\ref{main}. In this expository article, unlike~\cite{vdwproof}, we make a beeline towards proving just Theorem~\ref{main}, van der Waerden's Conjecture, in full.  
Reading this shorter exposition may also be useful as a precursor to reading the more general and more detailed article \cite{vdwproof}. 

\section{Preliminaries}

\subsection{Known results for intransitive and imprimitive groups}

That $N_n(G,H)=O(H^{n-1})$ holds for {\bf intransitive} groups $G$ was already shown by van der Waerden,  using the fact that polynomials having such Galois groups are exactly those that factor over~$\Q$.  In fact, an exact asymptotic of the form 
\begin{equation*}
\sum_{G\subset S_n\,\mathrm{intransitive}} N_n(G,H)= c_n H^{n-1}+O(H^{n-2})
\end{equation*} 
for an explicit constant $c_n>0$ was obtained by Chela~\cite{Chela}.  

\pagebreak 
Meanwhile, 
Widmer~\cite{Widmer} has given excellent bounds in the case of {\bf imprimitive} Galois groups $G$,  using the fact that polynomials having such Galois groups are exactly those that correspond to number fields having~a~nontrivial subfield.  (A permutation group $G$ is said to be {\it primitive} if it does not preserve any nontrivial partition of $\{1,\ldots,n\}$, and is {\it imprimitive} otherwise.)  Specifically, Widmer proves that 

\begin{equation*}
\sum_{G\subset S_n\,\mathrm{transitive\: but\: imprimitive}} N_n(G,H)= O(H^{n/2+2}).\vspace{.025in}
\end{equation*} 
Chow and Dietmann~\cite{CD} showed that van der Waerden's Conjecture holds for $n\leq 4$. Hence, to prove Theorem~\ref{main}, it suffices to show that $N_n(G,H)=O(H^{n-1})$ for {\bf primitive} permutation groups $G\neq S_n$ for all $n\geq 5$. 

\subsection{Primitive Galois groups that are not $S_n$} 

We now use the following result of Jordan on primitive permutation groups.

\begin{proposition}[Jordan] {If $G\subset S_n$ is a primitive permutation group on $n$ letters  that contains a transposition,  then $G=S_n$.} 
\end{proposition}

\begin{proof}
Suppose that $G\subset S_n$ is a primitive permutation group on $n$ letters containing a transposition.  Define an equivalence relation $\sim$ on $\{1,\ldots,n\}$ by defining $i\sim j$ if the transposition $(i\,j)\in G$.
 Then the action of $G$ clearly preserves the equivalence relation $\sim$ on $\{1,\ldots,n\}$. However, since $G$ is primitive, it cannot preserve any nontrivial partition of~$\{1,\ldots,n\}$. 
 Therefore, we must have $i\sim j$ (i.e., $(i\,j)\in G$) for all $i,j$, and so $G=S_n$ since~$S_n$ is generated by its transpositions.
\end{proof}

\noindent
Hence a primitive permutation group $G\neq S_n$ cannot contain a transposition.  This has the following consequence for the discriminants of polynomials $f\in\Z[x]$ of degree $n$ whose associated Galois group is not $S_n$:

\begin{corollary}\label{valindcon} Let $f$ be an integer polynomial of degree $n$, and let $K_f:=\Q[x]/(f(x))$.  \linebreak If $\Gal(f)\neq S_n$ is primitive, then the discriminant $\Disc(K_f)$ is squarefull.
\end{corollary}

\begin{proof}
The Galois group $G=\Gal(f)$ acts on the $n$ embeddings of $K_f$ into its Galois closure. 
Suppose $p\mid \Disc(K_f)$, and $p$ factors in $K_f$ as $\prod P_i^{e_i}$, where $P_i$ has residue field degree~$f_i$. 
If~$p$ is tamely ramified in $K_f$, then 
any generator $g\in G\subset S_n$ of an inertia group $I_p\subset G$ at~$p$ is the product of disjoint cycles consisting of $f_1$ cycles of length $e_1$, $f_2$ cycles of length $e_2$, etc. Since $G$ does not the contain a transposition, we must have $e_i>2$ for some $i$ or $e_i=2$ and $f_i>1$ for some $i$, or $e_i=e_j=2$ for some~$i\neq j$; thus the discriminant valuation $v_p(\Disc(K_f))=\sum (e_i-1) f_i$ is at least $2$ in that case. 
If $p$ is wildly ramified, then 
automatically the discriminant valuation $v_p(\Disc(K_f))$ is at least $2$. 
Therefore, $\Disc(K_f)$ is~squarefull. 
\end{proof}

\section{Proof of van der Waerden's Conjecture (Theorem~\ref{main})} 

We first prove the ``weak version'' of the conjecture,  namely, that $E_n(H)=O_\epsilon(H^{n-1+\epsilon})$.

To accomplish this, we divide the set of irreducible  monic integer polynomials $f(x)=x^n+a_1x^{n-1}+\cdots+a_n$, such that $|a_i|<H$ for all $i$ and $\Gal(f)<S_n$ is primitive, into three subsets.   
Let again $K_f:=\Q[x]/(f(x))$.  

We consider the following three cases:

\begin{itemize}

\item {\bf Case I:} The product $C$ of the ramified primes in $K_f$ is at most $H$, but the absolute discriminant $D=|\Disc(K_f)|$ is greater than $H^2$. 

\item {\bf Case II:} 
The absolute discriminant $D=|\Disc(K_f)|$ is at most $H^2$. 

\item {\bf Case III:} The product $C$ of the ramified primes in $K_f$ is greater than $H$.

\end{itemize}
We estimate the sizes of each of these sets in turn.

\subsection{{{\bf Case I:} $C\leq H$ and $D>H^2$}} 

We first consider those $f$ for which the product $C$ of ramified primes in $K_f:=\Q[x]/(f(x))$ is at most $H$, but the absolute discriminant $D$ of $K=K_f$ is greater than $H^2$.

By Corollary~\ref{valindcon}, $D$ is squarefull as we have assumed that $\Gal(f)<S_n$ is primitive.  Given such a $D$, the polynomials $f$ such that $|\Disc(K_f)|=D$ satisfy congruence conditions modulo $C=\mathrm{rad}(D)$ of density $O(\prod_{p\mid C} c/p^{v_p(D)})=O(c^{\omega(D)}/D)$ for a suitable constant $c>0$. 
Since $C<H$, the number of such $f$ can be counted directly within the box $\{|a_i|<H\}$ of sidelength $H$; we immediately have the estimate $O(H^nc^{\omega(D)}/D)$ for the number of such $f$. 

Summing $O(H^nc^{\omega(D)}/D)$ over all squarefull $D>H^2$ gives the desired estimate $O_\epsilon(H^{n-1+\epsilon})$ in this case: \vspace{-.05in} 

\begin{equation}\label{sumoff}
\sum_{D>H^2 \:\mathrm{ squarefull}} O(H^nc^{\omega(D)}/D)= O_\epsilon(H^{n-1+\epsilon}).
\end{equation}

\subsection{{{\bf Case II:} 
$D\leq H^2$}} 

\vspace{-.095in}
We next consider those $f$ for which 
the absolute discriminant $D$ of $K_f$ is at most $H^2$. 

The number of isomorphism classes of number fields $K=K_f$ of degree $n$ and absolute discriminant at~most~$H^2$ is $O((H^2)^{(n+2)/4})=O(H^{(n+2)/2})$, by a result of Schmidt~\cite{Schmidt}.\footnote{\noindent Improved results for $n$ sufficiently large have been obtained by Ellenberg and Venkatesh~\cite{EV}, Couveignes~\cite{Couveignes}, and most recently, Lemke Oliver and Thorne~\cite{LOT2}.}  For all $n>2$, Schmidt's estimate was recently improved to $O(H^{(n+2)/2-\kappa_n})$ for a small $\kappa_n>0$ in joint work with Shankar and Wang~\cite{BSW} (see also the work of Anderson, Gafni, Hughes, Lemke Oliver, and Thorne~\cite{aimgroup2}). By a result of Lemke Oliver and Thorne~\cite{LOT}, each isomorphism class of number field $K$ of degree $n$ can arise for at most $O(H\log^{n-1}H/|\Disc(K)|^{1/(n(n-1))})$ monic integer polynomials~$f$ of degree $n$.  
 
 Thus the total number of $f$ that arise in this case is at most  
 $$O(H^{(n+2)/2-\kappa_n}\cdot H\log^{n-1}H)=O(H^{n-1})\pagebreak $$when $n\geq 6$.  The exact asymptotic results known for the density of discriminants of quintic fields~\cite{dodpf} immediately gives $O(H^2\cdot H\log^4H)=O(H^{n-1})$ when $n= 5$ as well. 

\subsection{{{\bf Case III:} $C> H$}}
 
 \vspace{-.05in}
Finally, we consider those $f$ for which the product $C$ of ramified primes in $K_f$ is greater than $H$. 

Fix such an $f$.  By Corollary~\ref{valindcon}, for every prime $p \mid C$, the polynomial $f$ has  either at least a  triple root or at least a pair of double roots modulo $p$.  Therefore, changing $f$ by a multiple of $p$ does not change the fact that $p^2\mid \Disc(f)$.  (We thus say that ``$\Disc(f)$ is a multiple of~$p^2$ for mod~$p$ reasons'' in this case.) 

\begin{proposition}\label{modpreasons}
If $h(x_1,\ldots,x_n)$ is an integer polynomial,  such that $h(c_1,\ldots,c_n)$ is a multiple of $p^2$,  and indeed $h(c_1+pd_1,\ldots,c_n+pd_n)$ is a multiple of $p^2$ for all $(d_1,\ldots,d_n)\in\Z^n$,  then $\frac\partial{\partial x_n} h(c_1,\ldots,c_n)$ is a multiple of $p$.
\end{proposition}

\begin{proof}
Write $h(c_1,\ldots,c_{n-1},x_n)$ as  $$h(c_1,\ldots,c_n) + \textstyle\frac\partial{\partial x_n} h(c_1,\ldots,c_n)(x_n-c_n)+(x_n-c_n)^2r(x)\,.$$ 
Since the first and third terms are multiples of $p^2$ whenever $x_n\equiv c_n$ (mod $p$), the second term must be a multiple of $p^2$ as well, implying that $\frac\partial{\partial x_n} h(c_1,\ldots,c_n)$ is a multiple of~$p$. 
\end{proof}

Applying Proposition~\ref{modpreasons} to $h(a_1,\ldots,a_n)=\Disc(f)$, where $f(x)=x^n+a_1x^{n-1}+\cdots+a_n$, \linebreak we immediately conclude that $\frac\partial{\partial a_n} \Disc(f)$ is a multiple of $C$. Since both $\Disc(f)$ and $\frac\partial{\partial a_n} \Disc(f)$ are multiples of $C$ for our choice of $f$, the resultant of $\Disc(f)$ and $\frac\partial{\partial a_n} \Disc(f)$, i.e., the {\it double discriminant} 
$$\mathrm{DD}(a_1,\ldots,a_{n-1}):=\Disc_{a_n}(\Disc_x(f(x))),$$ 
must also be a multiple of $C$ for this choice of $f$. (An examination of the polynomial $f(x)=x^n+a_{n-1}x+a_n$ shows that $\mathrm{DD}(a_1,\ldots,a_{n-1})$ does not identically vanish.)

 
Let $f\in\Z[x]$ be a polynomial for which the product $C$ of ramified primes in $K_f$ is greater than $H$. 
For such an $f$, we have proven that the polynomial $\mathrm{DD}(a_1,\ldots,a_{n-1})$ is a multiple of $C$. 
The number of possible $a_1,\ldots,a_{n-1}\in[-H,H]^{n-1}$ such that $$\mathrm{DD}(a_1,\ldots,a_{n-1})=0$$is $O(H^{n-2})$,  and so the total number of $f$ with $\mathrm{DD}(a_1,\ldots,a_{n-1})=0$ is $O(H^{n-1})$. 

Let us now fix $a_1,\ldots,a_{n-1}$ such that $\mathrm{DD}(a_1,\ldots,a_{n-1})\neq 0$.  Then 
$\mathrm{DD}(a_1,\ldots,a_{n-1})$ has at most $O_\epsilon(H^\epsilon)$ factors $C>H$. 
Once $C$ is determined by $a_1,\ldots,a_{n-1}$,  the number of solutions for $a_n$ (mod $C$)~to $$\Disc(f)\equiv0 \mbox{ (mod $C$)}$$ is $$(\deg_{a_n}(\Disc(f)))^{\omega(C)}=O_\epsilon(H^\epsilon).$$
 Since $C>H$, the number of $a_n\in[-H,H]$ is also $O_\epsilon(H^\epsilon)$,  and so the total number  of $f$ in this case is again  $O_\epsilon(H^{n-1+\epsilon})$.

\subsection{Conclusion}

We have thus proven the following theorem:

\begin{theorem}\label{weakversion}
Let $E_n(H)$ denote the number of monic integer polynomials $f(x)=x^n+a_1 x^{n-1} +\cdots+a_n$ of degree $n$ with $|a_i|\leq H$ for all $i$ such that $\Gal(f)$ is not $S_n$.  
Then $E_n(H)=O_\epsilon(H^{n-1+\epsilon})$.
\end{theorem}

%
%
%
%
%
%

\subsection{Removing the $\epsilon$}

Removing the $\epsilon$ in Theorem~\ref{weakversion} turns out to be just as much work as proving Theorem~\ref{weakversion}. 

To remove the $\epsilon$ in {Case I}, we replace the condition $$C\leq H \mbox{ and } D>H^2$$ by  $$C\leq H^{1+\delta} \mbox{ and } D>H^{2+2\delta}$$ for some small $\delta=\delta_n>0$.  

The resulting congruence conditions in Case I are now modulo an integer $C$ that is potentially larger than the sidelength $H$ of the box.  However, using Fourier analysis (see~Subsection~3.6), we show sufficient equidistribution of the residue classes modulo $C$ that we are counting to extend the validity
of the count $O(H^nc^{\omega(D)}/D)$ even when $C < H^{1+\delta}$.  Specifically, using Fourier analysis, we prove:

\begin{lemma}\label{keylemma}
Let $0<\delta < 1/(2n-1)$. For each $i=1,\ldots,m$, let $k_i>1$ be a positive integer. 
Let $D$ be a positive integer with prime factorization $D=p_1^{k_1}\cdots p_m^{k_m}$ such that 
$C=p_1\cdots p_m<H^{1+\delta}$.
Then the number of monic integer polynomials $f$ of degree $n$ 
in $[-H,H]^{n}$ such that $|\Disc(K_f)|=D$ 
is at most $O(c^{\omega(C)} H^{n}/D).$
\end{lemma}
The estimate $O(H^{n}c^{\omega(C)} /D)$ of Lemma~\ref{keylemma}, summed over all squarefull $D>H^{2+2\delta}$, then gives $O(H^{n-1-\delta+\epsilon})$; this thereby removes the $\epsilon$ in (\ref{sumoff}). 

We now replace the condition $$D\leq H^2$$ in {Case II} by $$D\leq H^{2+2\delta}.$$ 
Since we had already proven a power saving in this case, the total estimate in this case, even with this small change, is still $O(H^{n-1})$.

%



%
%
%


Finally, we turn to Case III, and replace the condition 
$$C>H$$
by
$$C>H^{1+\delta}.$$
Thus Cases I, II, and III again cover all possibilities.

\pagebreak
Note that there are two sources of the $\epsilon$ in our original treatment of {Case~III}:

\begin{enumerate}
\item[(i)]
The first source is that the number of factors $C$ of $\mathrm{DD}(a_1,\ldots,a_{n-1})$ is~$O(H^\epsilon)$. 
\item[(ii)]
The second source is that, for each choice of $a_1,\ldots,a_{n-1}$ such that $\mathrm{DD}(a_1,\ldots,a_{n-1})\neq 0$, and each choice of factor $C
\mid \mathrm{DD}(a_1,\ldots,a_{n-1})$,  there are $O((n-1)^{\omega(C)})=O(H^\epsilon)$
 choices for $a_n$. 
\end{enumerate}


%
%

We remove the $\epsilon$'s in these arguments as follows.  We choose a suitable factor $C'$ of~$C$ that is between $H^{1+\delta/2}$ and $H^{1+\delta}$ in size,  in which case we can handle it by a method analogous to  {Case I} (with $C'$ in place of~$C$).
 Otherwise, we can choose a factor $C'>H$ of~$C$ all of whose prime divisors are greater than $H^{\delta/2}$,  in which case we can handle it by the same method as in {Case III} above (with $C'$ in place of $C$)---but with no $\epsilon$ occurring because~$C'$ will have at most a bounded number of prime factors!


To be more precise, we break into two subcases:

\vspace{-.05in}
\subsubsection*{Subcase (i): $A=\displaystyle\prod_{{\scriptstyle p\mid C}\atop{\scriptstyle p>H^{\delta/2}}} p \leq H$}

\vspace{-.05in} 
In this subcase, $C$ has a factor $B$ between $H^{1+\delta/2}$ and $H^{1+\delta}$, with $A\mid B\mid C$. Let $B$ be the largest such factor. 
Let $D':=\prod_{p\mid B} p^{v_p(D)}$. Then $D'>H^{2+\delta}$.
We now carry out the argument of Case I, with $B$ in place of $C$, and $D'$ in place of $D$. (Note that $D$ determines $C$ determines $B$ determines $D'$.) 
Summing $O(c^{\omega(D')}H^n/D')$ over all squarefull $D'>H^{2+\delta}$ then gives the desired estimate $O(H^{n-1})$ in this subcase: 
$$\sum_{D'>H^{2+\delta} \:\mathrm{ squarefull}} O(c^{\omega(D')}H^n/D')= O_\epsilon(H^{n-1-\delta/2+\epsilon}) =O(H^{n-1}).$$

\vspace{-.15in}
\subsubsection*{Subcase (ii): $A=\displaystyle\prod_{{\scriptstyle p\mid C}\atop{\scriptstyle p>H^{\delta/2}}} p > H$}

\vspace{-.05in} 
In this subcase, we carry out the original argument of Case III, with $C$ replaced by~$A$.
We have $A\mid \mathrm{DD}(a_1,\ldots,a_{n-1}):=\Disc_{a_n}(\Disc_x(f(x)))$. 

Fix $a_1,\ldots,a_{n-1}$ such that $\mathrm{DD}(a_1,\ldots,a_{n-1})\neq 0$.   Being bounded above by a fixed power of $H$, we see that 
$\mathrm{DD}(a_1,\ldots,a_{n-1})$ can have at most a {\bf bounded} number of possibilities for the factor $A$ (since all prime factors of $A$ are bounded below by a fixed positive power of $H$)!

Once $A$ is determined by $a_1,\ldots,a_{n-1}$,   then the number of solutions for $a_n$ (mod $A$) to $$\Disc(f)\equiv0 \mbox{ (mod $A$)}$$ is $$O((n-1)^{\omega(A)})=O(1).$$ 
 Since $A>H$, the total number  of $f$ in this subcase is also   $O(H^{n-1})$.
 
 \vspace{.1in}
 This completes the proof of Theorem~\ref{main}, assuming the truth of Lemma~\ref{keylemma}.

\subsection{Proof of Lemma~\ref{keylemma}}
 
For a ring $R$, let $V^1_R$ denote the set of monic polynomials of degree $n$ over $R$, which we may identify with $R^n$.  For a function $\Psi_q : V_{\Z/q\Z}^1 \rightarrow \C$, let \smash{$\widehat{\Psi_q}\colon V^{1\ast}_{\Z/q\Z} \rightarrow \C$} be its Fourier transform defined by the usual formula
$$
{\smash{\widehat{\Psi_q}(g)} = \frac1{q^{n}}
\sum_{f \in V^{1}_{\Z/q\Z}} \Psi_q(f) \exp\left(\frac{2 \pi i [f, g]}{q}\right),}
$$
where $V^{1*}_R$ denotes the $R$-dual of $V^1_R$. 
If $\Psi_q$ is the characteristic function of a set {$S\subset V^1_{\Z/q\Z}$}, then upper bounds on the maximum $M({\Psi}_q)$ of \smash{$|\widehat{\Psi}_q(g)|$} over all nonzero $g$ constitutes a measure of 
{equidistribution} of {$S$} in suitable boxes of monic integer polynomials of degree $n$. This is because, for any Schwartz function {$\phi$} approximating the characteristic function of the box $[-1,1]^{n}$, the twisted Poisson summation formula gives 
{\begin{equation}
\begin{array}{rl}\label{twist2}
& \displaystyle\sum_{f=(a_1,\ldots,a_n) \in V_\Z^1} \Psi_q(a_1,\ldots,a_n) \phi(a_1/H,a_2/H,\ldots,a_n/H) \\[.325in] = 
H^{n} &\displaystyle\sum_{g=(b_1,\ldots,b_n) \in V^{1*}_\Z} \widehat{\Psi_q}\,(b_1,\ldots,b_n) \widehat{\phi}\left(b_1H/q,b_2H/q\ldots,b_nH/q\right).
\end{array}\end{equation}}For suitable {$\phi$}, the left side of (\ref{twist2}) will be an upper bound for the number of 
elements in~$S$ in the box $[-H,H]^{n}$. The  {$g = 0$} term is the expected main term,  while the rapid decay of \smash{$\widehat{\phi}$} implies that the error term is effectively bounded by $H^{n}$ times the sum of \smash{$|\widehat{\Psi_q}(g)|$} 
over all $0\neq g\in V^*_\Z$ whose coordinates are of size at most $O(q^{1+\epsilon}/H)$, and this in turn can be bounded by $O(H^{n}(q^{1+\epsilon}/H)^{n}M(\Psi_q))=O(q^{n+\epsilon}M(\Psi_q))$. 


In this subsection, we show that the monic polynomials of degree $n$ over $\F_p$, having splitting type containing a given splitting type $\sigma$, are very well distributed in boxes. We accomplish this by demonstrating cancellation in the Fourier transform of certain corresponding weighted characteristic functions, using Weil's bounds~\cite{Weil} on exponential sums.

To state the result precisely, we shall need the following definitions. 

\begin{itemize}
\item If a monic polynomial $f$ (over $\Z$, or over $\F_p$) factors modulo $p$ as $\prod_{i=1}^r P_i^{e_i}$, with $P_i$ monic irreducible and $\deg(P_i)=f_i$, then the {\it splitting type} $(f,p)$ of $f$ is defined to be $(f_1^{e_1}\cdots f_r^{e_r})$. 

\item The {\it index} $\ind(f)$ of $f$ modulo~$p$ (or the {\it index} of the splitting type $(f,p)$ of~$f$) is then defined to be $\sum_{i=1}^r (e_i-1)f_i$. 

\item More abstractly, we call any expression $\sigma$ of the form $(f_1^{e_1}\cdots f_r^{e_r})$ a {\it splitting type}.


\item 
The {\it index} $\ind(\sigma)$ of $\sigma$ is defined to be $\sum_{i=1}^r (e_i-1)f_i$. 

\item 
Finally, $\#\Aut(\sigma)$ is defined to be $\prod_i f_i$ times the number of permutations of the factors
$f_i^{e_i}$ that preserve $\sigma$. (See~\cite[\S2]{mass} for the motivation for this definition.)
\end{itemize}

\pagebreak 
\begin{proposition}\label{ftgen2}
Let $\sigma=(f_1^{e_1}\cdots f_r^{e_r})$ be a splitting type with   
$\ind(\sigma)=k$.
Let $w_{p,\sigma}:V^1_{\F_p}\to\C$ be defined by
\begin{align*}
w_{p,\sigma}(f):=&\text{ the number of $r$-tuples $(P_1,\ldots,P_r)$, up to the action of the group of}  \\[-.04in] & \text{  permutations of $\{1,\ldots,r\}$ preserving $\sigma$, such that the $P_i$ are distinct } \\[-.04in] &\text{ irreducible monic polynomials with $\deg P_i=f_i$ for each $i$ and $P_1^{e_1}\cdots P_r^{e_r}\mid f$}.
\end{align*}
Then 
\begin{equation*}
\widehat{w_{p,\sigma}}(g)=
\begin{cases}
	{\displaystyle\frac{p^{-k}}{\small\Aut(\sigma)}\!+\!O(p^{-k-1})}	& \!\text{if } {g=0};\\[.2in]
	{O(p^{-k-1/2})}		& \!\text{if $g\neq 0$.}
\end{cases}
\end{equation*}
\end{proposition}

\begin{proof}
We have
\begin{eqnarray}\displaystyle
\widehat{w_{p,\sigma}}(g)&=&\displaystyle\frac{1}{p^{n}}\sum_{f\in V^1_{\F_p}} e^{2\pi i[f,g]/p} w_{p,\sigma}(f)
\label{firstexp2}\\
&=&\displaystyle\frac{1}{p^{n}}\sum_{P_1,\ldots,P_r}
\:\sum_{P_1^{e_1}\cdots P_r^{e_r} \mid f}
e^{2\pi i[f,g]/p}. \label{secondexp2}
\end{eqnarray}

When $g=0$, evaluating (\ref{secondexp2}) gives 
$\widehat{w_{p,\sigma}}(0)=(p^{-k}/\#\Aut(\sigma))+O(p^{-k-1})$.  This is
because 1) the number of possibilities for $P_1,\ldots,P_r$ is 
$(1/\#\Aut(\sigma))p^{\sum {e_i}}+O(p^{{\sum e_i-1}})$, and 2) the number of $f\in V^1_{\F_p}$ such that $P_1^{e_1}\cdots P_r^{e_r}$ divides $f$ is $p^{n-\sum {e_if_i}}.$ We conclude that
$$\widehat{w_{p,\sigma}}(0)=\frac1{p^n}\left(\frac{p^{\sum {e_i}}}{\#\Aut(\sigma))}+O(p^{\sum {e_i}-1})\right)p^{n-\sum {e_if_i}} = {\displaystyle\frac{p^{-k}}{\small\Aut(\sigma)}+O(p^{-k-1})}.	$$

When $g\neq 0$, we apply the Weil bound~\cite{Weil} on exponential sums to establish cancellation in and thereby obtain a nontrivial estimate on (\ref{secondexp2}) as follows. As already noted, the total number of $f$ (counted with multiplicity) in the double sum in (\ref{secondexp2}) is $\asymp\frac{p^{n-k}}{\#\Aut(\sigma)}$. We partition these polynomials  $f(x)$ into orbits of size $p$ under the action of translation $x\mapsto x+c$ for $c\in\F_p$.  We then consider the elements of each orbit together in (\ref{secondexp2}). 
Given such a polynomial $f(x)$, if $g\neq 0$ and $p>n$, then $[f(x+c),g]$ is a nonconstant univariate polynomial $Q(c)$ in $c$ of degree at most $n$. 
In that case, the contribution in (\ref{secondexp2}) corresponding to $f(x)$ and its translates $f(x+c)$ add up to 
$$\sum_{c\in\F_p} e^{2\pi i [f(x+c),g]} = \sum_{c\in\F_p} e^{2\pi i Q(c)}$$
which is at most $(n-1)p^{1/2}$ in absolute value by the Weil bound.  
Summing over the $O(p^{n-k-1})$ equivalence classes of these $f(x)$ under the action of translation $x\mapsto x+c$ then yields
$$|\widehat{w_{p,\sigma}}(g)|=O(p^{-n}p^{n-k-1}p^{1/2})=O(p^{-k-1/2}),$$
which improves upon the trivial bound $O(p^{-k}),$ as desired.
\end{proof}



\begin{corollary}\label{Dequi2}
Let $D$ be a positive integer with prime factorization $D=p_1^{k_1}\cdots p_m^{k_m}$ and let $C=p_1\cdots p_m$. 
The number of integral monic polynomials of degree $n$ in $[-H,H]^{n}$ 
that modulo $p_i$ have index at least $k_i$ for $i=1,\ldots,m$ is $O(c^{\omega(C)}H^{n}/D)+O_\epsilon(C^{n-1/2+\epsilon}/D)$. 
\end{corollary}


\begin{proof}
First, we note that the values of the $\Z/C\Z$-Fourier transform are simply products of values of the $\F_{p_i}$-Fourier transforms (one value for each $i$). 

Let $\phi$ be a smooth function with compact support that is identically $1$ on $[-1,1]^{n}$.  Let $\Psi:V^1_{\Z/C\Z}\to \R$ be defined by $\Psi=\prod_i(\sum_{\sigma:\ind(\sigma)\geq k}w_{p_i,\sigma}$). 
By twisted Poisson summation~(\ref{twist2}), we have
\begin{eqnarray*}
&\!\!\!\!\!\!\!\!\!\!\!\!\!\! & 
\!\!\sum_{f \in V_\Z^1} \Psi(f) \phi(f/H)\\[.025in]  &\!\!\!\!\!\!\!\!\!\!\!\!=\!\!\!\!& \!\!H^{n} \sum_{g \in V^{1*}_\Z} \widehat{\Psi}(g) \widehat{\phi}\left(\frac{g H}{C} \right) \\
&\!\!\!\!\!\!\!\!\!\!\!\!\ll\!\!\!\!& \!\!H^{n}\widehat{\Psi}(0) \widehat{\phi}(0) + H^{n}\!\hspace{-.2in} \sum_{g \in \bigl[-\frac{C^{1+\epsilon}}{H},\frac{C^{1+\epsilon}}{H}\bigr]^{n}\cap V^{1*}_\Z\setminus\{0\}} \hspace{-.2in} \!\!\!\left|\widehat{\Psi}(g)\right|+ H^{n} \hspace{-.2in}
\!\!\sum_{g\notin \bigl[-\frac{C^{1+\epsilon}}{H},\frac{C^{1+\epsilon}}{H}\bigr]^{n}\cap V^{1*}_\Z}\hspace{-.2in}\!\!\!\left|\widehat{\phi}\left(\frac{g H}{C} \right)\right| \\[.1in]
&\!\!\!\!\!\ll_{\epsilon,N}\!\!\!\!& \,c^{\omega(C)}H^{n}/D +
H^{n}\hspace{-.2in} \!\sum_{g \in \bigl[-\frac{C^{1+\epsilon}}{H},\frac{C^{1+\epsilon}}{H}\bigr]^{n}\cap V^{1*}_\Z\setminus\{0\}} \!\hspace{-.2in} \!\!\!\left|\widehat{\Psi}(g)\right|
+ H^{n} \!\!\!\!\!\!\!\!\!\sum_{g\notin  \bigl[-\frac{C^{1+\epsilon}}{H},\frac{C^{1+\epsilon}}{H}\bigr]^{n}\cap V^{1*}_\Z} \!\hspace{-.2in}\!\! \left(\frac{\|g\| H}{C} \right)^{\!-N} 
\end{eqnarray*}
for any integer $N$; the bound on the third summand holds because $\phi$ is smooth and thus is $N$-differentiable for any integer $N$, and so $\widehat{\phi}(g)\ll_N \|g\|^{-N}$ (see, e.g., \cite[Chapter~5 (Theorem~1.3)]{SS}). 
By choosing $N$ sufficiently large, the third term can be absorbed into the first term.
We now estimate the second term using Proposition~\ref{ftgen2}:
\begin{equation*}
\begin{array}{rcl}
\displaystyle H^{n}\hspace{-.175in}  \sum_{g\in\bigl[-\frac{C^{1+\epsilon}}{H},\frac{C^{1+\epsilon}}{H}\bigr]\cap V_\Z^{1*}\backslash\{0\}} \hspace{-.2in} \left|\widehat{\Psi}(g)\right|
&\!\ll\!&\displaystyle
H^{n}c^{\omega(C)}\sum_{\substack{q\mid C}}\sum_{\substack{g\in\bigl[-\frac{C^{1+\epsilon}}{H},\frac{C^{1+\epsilon}}{H}\bigr]^n\cap V_\Z^{1*}\backslash \{0\}\\(\ct(g),C)=q}}q^{1/2}\prod_{i=1}^m p_i^{-k_i-1/2}
\\[.425in]
&\!\ll_\epsilon\!&\displaystyle
H^{n}\sum_{\substack{q\mid C}}
\frac{C^{n+\epsilon}}{q^{n}H^{n}}\cdot q^{1/2}\prod_{i=1}^m p_i^{-k_i-1/2}
\\[.375in]
&\!\ll_\epsilon\!&\displaystyle
C^\epsilon \prod_{i=1}^m p_i^{n-k_i-1/2}
\\[.375in]
&\!\ll_\epsilon\!&\displaystyle
C^{n-1/2+\epsilon}/D,
\end{array}
\end{equation*}
where the content $\ct(g)$ of $g$ denotes the largest integer such that $g/\ct(g)\in V_\Z^{1*}$.  This yields the desired result.
\end{proof}

\pagebreak 
We now complete the proof of the key lemma, Lemma~\ref{keylemma}. Suppose $f$ is a monic integer polynomial of degree $n$ such that $|\Disc(K_f)|=D=\prod p_i^{k_i}$.  Then (aside from primes $p\mid n$ where there may be wild ramification) the index of $f$ (mod $p_i$) is at least $k_i$.
By  Corollary~\ref{Dequi2}, the number of monic integer polynomials $f$ of degree $n$ such that the index of $f$ (mod $p_i$) is at least $k_i$ for all $i$, and the product $C=\prod p_i$ of ramified primes in $K_f$ satisfies $C<H^{1+\delta}$, is $$O(c^{\omega(C)}H^{n}/D)+O_\epsilon(C^{n-1/2+\epsilon}/D)
=O(c^{\omega(C)}H^{n}/D)+O_\epsilon((H^{1+\delta})^{n-1/2+\epsilon}/D)
=O(c^{\omega(C)} H^{n}/D)$$ since $\delta<1/(2n-1)$.  This completes the proof of Lemma~\ref{keylemma} (and thus also Theorem~\ref{main}). $\Box$

\section{Related results and variations} 

We note that the non-monic case (the subject of van der Waerden's original conjecture) can be handled in essentially the same way, in order to prove that the number of integer-coefficient polynomials of degree $n$ with coefficients bounded in absolute value by $H$ whose Galois group is not $S_n$ is $E_n^\ast(H)=O(H^{n})$.

Other results that can be proven using extensions of the methods described in this article:
\begin{itemize}
\item If $G\neq S_n$ or $A_n$: 
\begin{itemize} \item If $n\geq 10$,  then $N_n(G,H)=O(H^{n-2})$;  \item If $n\geq28$, then $N_n(G,H)=O(H^{n-3})$;  \item For sufficiently large $n$, we have $E_n(H)=O(H^{n-cn/\log^2n})$ where $c$ is an absolute constant.\end{itemize}
\item For $p$ a prime, if $G=C_p$ (the cyclic group of order $p$), then $N_n(C_p,H)=O(H^2)$.  
\item If $G$ is a regular permutation group on $n$ letters, then 
$N_n(G,H)=O(H^{3n/11\,+\,1.164}).$
\item We have $N_{11}(M_{11},H)=O(H^{8.686})$, where $M_{11}$ is the Mathieu group on 11 letters. 
\item (A question of Serre)  The number of monic even integer polynomials $$g(x)=x^{2n}+a_1x^{2n-2}+a_2x^{2n-4}+\cdots+a_n$$ with $|a_i|<H$ for all $i$ whose Galois group is not the Weyl group $W(B_n)\cong S_2^n\rtimes S_n$  is~$\:\asymp H^{n-1/2}$, 
and the number for which it is also not the Weyl group $W(D_n)$ is  $\:\asymp H^{n-1}$. 
\end{itemize}
For more details on these results and variations, see~\cite{vdwproof}. 

\subsection*{Acknowledgments}

We are extremely grateful to Benedict Gross, Danny Neftin, Andrew O'Desky, Robert Lemke~Oliver, Ken~Ono, Fernando Rodriguez-Villegas, Arul Shankar, Don Zagier, and the anonymous referee for all their helpful comments and encouragement. 
Most of all, we thank Don~Zagier for his friendship and years of inspiration---wishing him a very happy birthday and many happy returns!

\end{document}